\documentclass[11pt]{amsart}
\usepackage{amssymb,amsxtra,amsmath,amsfonts}
\usepackage{verbatim}
\usepackage{graphicx,color,graphics} 
\usepackage{xypic}
\usepackage{braket}
\usepackage[colorlinks = true,
   linkcolor = blue,
 urlcolor  = green,
citecolor = magenta,
anchorcolor = blue]{hyperref}
\theoremstyle{plain}

\newtheorem*{theorem*}{Theorem}
\newtheorem{theorem}{Theorem}

\newtheorem{lemma}[theorem]{Lemma}
\newtheorem{proposition}[theorem]{Proposition}

\theoremstyle{definition}
\newtheorem{definition}[theorem]{Definition}

\theoremstyle{remark}
\newtheorem{remark}[theorem]{Remark}

\numberwithin{equation}{section}
\numberwithin{theorem}{section}

\renewcommand{\comment}[1]{}
\def\CC{\mathbb{C}}

\def\Y{\mathcal{Y}}

\def\F{\mathcal{F}}

\def\ZZ{\mathbb{Z}}

\DeclareMathOperator\End{End}

\DeclareMathOperator\Hom{Hom}

\DeclareMathOperator\ad{ad}

\DeclareMathOperator\Span{span}

\def\vac{{\boldsymbol{1}}}  

\def\lieg{{\mathfrak{g}}}
\def\lieh{{\mathfrak{h}}}

\begin{document}

\title[Lambda bracket and Intertwiners]{Lambda bracket and Intertwiners}
\author{Juan J. Villarreal}
\address{Department of Mathematical Sciences, 
University of Bath, 
Bath BA2 7AY,
United Kingdom}
\email{juanjos3villarreal@gmail.com}
\maketitle
\begin{abstract}We describe the intertwiners between modules of a vertex algebra using the language of lambda bracket. We apply this formalism to obtain some classical results on conformal field theory. 



\end{abstract}


\section{Introduction}

It is well known that vertex algebras \cite{Bo} can be equivalently formulated in terms of a lambda bracket, \cite{K,BK}. This equivalence extends naturally to modules of vertex algebras. 


In this work, the intertwiners between modules of a vertex algebra \cite{FHL} are equivalently formulated in terms of a lambda bracket. Our motivation is to obtain a formalism useful to do explicit calculations with intertwiners. 

In more detail,  we consider a {formal Fourier transform} defined linearly by its action on  monomials, see Section \ref{s1} for more details, as follows
\begin{equation}\label{fo}
F^{\lambda}_{z}\left({z^{-n-1}}\right)=\lambda^{(n)}:=\frac{1}{\Gamma(n+1)}\lambda^n\, ,  \qquad n\in \CC \, , 
\end{equation}
where ${1}/{\Gamma(x)}$ denotes the inverse of the Gamma function. 
This transformation naturally generalizes the formal Fourier transform in \cite{K},  see Lemmas \ref{lem0} and \ref{lem1}.

In this work, a module $M$ for a vertex algebra $V$ is assumed to have a translation operator $T^{M}\in \End(M)$, see Definition \ref{d3.11}. We will denote this as $T\in \End(M)$ by abuse of notation. 

 Let $M_1 , M_2, M_3$ be $V$-modules. For an  intertwiner $\Y$, see Definition \ref{defint}, we have for $a\in M_1, b\in M_2$ that
 \[\Y(a,z)=\sum_{n\in \CC}a_{(n)}b \/\, z^{-n-1}\,,  \, \quad a_{(n)}b\in M_3\, .\]
 
We define the  $\lambda$-bracket for intertwiners as follows 
\begin{equation}\label{fo2}[a_{\lambda}b]:=F^{\lambda}_{z}(\Y(a,z)b)\, .
\end{equation}
This bracket together with the endomorphisms $T$, Proposition \ref{pro3.2}, satisfy for $v\in V$ that 
\begin{enumerate}
\item[i)]\label{i)}  $[T a_{\lambda}b]=-\lambda [ a_{\lambda}b]$,    $\quad[ a_{\lambda}T b]=(\lambda+T) [ a_{\lambda}b]$$;$
\medskip
\item[ii)] \label{jac3} $[v_{\lambda}[a_{\mu}b]]=[a_{\mu}[v_{\lambda}b]]+\iota_{\mu, \lambda}[[v_{\lambda}a]_{\lambda+\mu}b]$.   
\vspace{0,08cm}       
\end{enumerate}

Additionally, we have a product $\cdot: M_1\otimes M_2\rightarrow M_3$ defined by
\begin{equation}\label{eqpro}a\cdot b:=a_{(-1)}b\, .\end{equation}
This product, see Proposition \ref{pro3.3}, satisfies
\begin{enumerate}
\medskip
\item[iii)] $(va)b-v(ab)=\left(\int^{T}_{0}d\lambda v\right)[a_{\lambda}b]^0+\left(\int^{T}_{0}d\lambda a\right)[v_{\lambda}b] $, 
\vspace{0,1cm}
\item[iv)] $v(ab)-a(vb)=\left(\int^{0}_{-T}[v_{\lambda}a]d\lambda\right) b$.
\medskip

\end{enumerate}
where $[a_{\lambda}b]^0:=F^{\lambda}_{z}(\sum_{n\in \ZZ}a_{(n)}b \/\, z^{-n-1})$. Finally, the $\lambda$-bracket and the product, see Proposition \ref{pro3.4}, are related as follows
\begin{enumerate}
\medskip
\item[v)] $[v_{\lambda}ab]=a[v_{\lambda}b]+[v_{\lambda}a]b+\int_{0}^{\lambda} [[v_{\lambda}a]_{\mu}b]^{0}d\mu$,
\medskip
\item[vi)] $[v  a_{\lambda}  b ]=(e^{T\partial_{\lambda}}v)[a_{\lambda}b]+(e^{T\partial_{\lambda}}a)[v_{\lambda}b]+\int_{0}^{\lambda}  [a_{\mu}[v_{\lambda-\mu}b]]d\mu $,
\medskip
\item[vii)] $[a_{\lambda}vb]=v[a_{\lambda}b]+[a_{\lambda}v]b+\int_{0}^{\lambda} [[a_{\lambda}v]_{\mu}b]d\mu$.
\medskip
\end{enumerate}
where $[a_{\lambda}v]:=-[v_{-\lambda-T}a]$, see also definition \eqref{defintegral}. We have the next result
\begin{theorem*} Let $M_1 , M_2, M_3$ be $V$-modules. A $\lambda$-bracket $[\cdot_{\lambda}\cdot]$ and a product $\cdot$ satisfying \textnormal{i)}, $\cdots$ , \textnormal{vii)} define an intertwiner.
\end{theorem*}
 For a more precise statement see Theorem \ref{theorem3.8}. We remark that besides characterized intertwiners the identities \textnormal{i)}, $\cdots$ , \textnormal{vii)}  also allow us to do calculations with intertwiners. 
 
 In Section \ref{s4}, we apply the identities above to obtain some classical results. First we consider the Virasoro algebra, we obtain some relations which goes back to \cite{BPZ} and more generally to \cite{FF,  FF1}.  Second for the integral levels of affine Kac-Moody algebras, we obtain some results which goes back to \cite{KZ, GW}.

 In future works, we will extend this work to logarithmic intertwiners \cite{Mi}, and study some of  the results in \cite{CR, CR2, Ad}.  Also, we will study an associated graded for intertwiners and related this with the $C_2$-algebra, \cite{Li3, A, A2}. 

This work is organized as follows:  In Section \ref{pre}, we introduce the definitions used in this work. In Section \ref{s1}, we prove the main results of this work,  which were briefly described above. Finally, in Section \ref{s4}, we consider some examples.

\section{Preliminars} \label{pre}

We denote by $V$ a vector space, and by $V [z]$ (respectively, $V [\![z]\!]$) the space of polynomials (respectively, formal power series) in $z$ with coefficients in $V$.

Let $\Gamma\subset \CC$ such that $\Gamma + \ZZ = \Gamma$ and $\Gamma/\ZZ$ is a finite subset of  $\CC/\ZZ$. We denote by
$V[\![z]\!]z^{-\Gamma}$ the space of infinite sums $\sum_{n}f_{n}z^{n}$, where $f_n\in V$ and $n$ runs over the union of finitely many sets of the form $\{-d_{i} + \ZZ_{\geq0}\}$ with $d_i \in  \Gamma$. In particular, $V[\![z]\!]z^{-\ZZ}$ is the space $V(\!(z)\!):=V[\![z]\!][z^{-1}]$ of formal Laurent series. 
Finally note that $V [\![z]\!]z^{-\Gamma}$ is equipped with the usual action of the derivative $\partial_{z}$. 



\begin{definition} \label{d3.11}
A vertex algebra is a vector space $V$ equipped with a vector $\vac\in V$  and a linear map 
\[Y\colon  V\to \Hom(V,V(\!(z)\!))\,, \qquad v\mapsto Y(v,z)=\sum_{n\in \ZZ}v_{(n)}z^{-n-1}\,, \]
subject to the following axioms:

\medskip
(\emph{i})\; $Y(\vac,z)=I_{V}$, \; $Y(v,z)\vac\in V[\![z]\!]$, \; $Y(v,z)\vac\big|_{z=0} = v$. 


(\emph{ii})\;  $\forall$ $u,v\in V$ and $n\in \ZZ$ 
\begin{equation}\label{b1}
\begin{split}
\iota_{z_1, z_2}z_{12}^{n}Y(v&,z_1)Y(u, z_2)-\iota_{z_2, z_1}z_{12}^{n}Y(u, z_2)Y(v,z_1)b\\
&=\sum_{i\geq 0}Y( v_{(n+i)}u_j, z_2) \partial^{(i)}_{z_2}\delta(z_1, z_2)\, .
\end{split}
\end{equation}
\end{definition}

On a vertex algebra we define $T\in \End(V)$ by $Tv=v_{(-2)}\vac$, then by definition above we have $[T,Y(v,z)] = D_z Y(v,z)$, see  \cite{K, FB}.


\begin{definition} \label{d3.11}
A $V$-module is a vector space $M$ equipped with an endomorphism $T\in \End(M)$ and a linear map 
\[Y^{M}\colon  V\to \Hom(M,M(\!(z)\!))\,, \qquad v\mapsto Y^{M}(v,z)=\sum_{n\in \ZZ}v_{(n)}z^{-n-1}\,, \]
subject to the following axioms:

\medskip
(\emph{i})\; $Y^{M}(\vac,z)=I_{M}$, $[T,Y^M(v,z)] = D_z Y^M(v,z)$;


(\emph{ii})\;  $\forall$ $u,v\in V$ and $n\in \ZZ$
\begin{equation}\label{b1}
\begin{split}
\iota_{z_1, z_2}z_{12}^{n}Y^M(v&,z_1)Y^M(u, z_2)-\iota_{z_2, z_1}z_{12}^{n}Y^M(u, z_2)Y^M(v,z_1)\\
&=\sum_{i\geq 0}Y^M( v_{(n+i)}u, z_2) \partial^{(i)}_{z_2}\delta(z_1, z_2)\, .
\end{split}
\end{equation}
\end{definition}




Additionally, we have the definition of intertwiner from \cite{FHL}


\begin{definition}\label{defint} Let $M_1 , M_2, M_3$ three $V$-modules. An intertwining operator of type $\binom{M_3}{M_1\quad M_2}$ is a linear map 
\[\Y:M_1 \rightarrow \Hom(M_2,M_3[\![z]\!]z^{-\Gamma})\,, \qquad a\mapsto \Y(a,z)=\sum_{n\in \CC}a_{(n)} \/\, z^{-n-1}\, , \]
subject to the following axioms:

\medskip

(\emph{i})\; 
$[T,\Y(a,z)]=\Y(Ta,z)  = D_z \Y(a,z)$.

(\emph{ii})\;  $\forall$ $v\in V, a\in M_1$ and $n\in \ZZ$
\begin{equation}\label{b}
\begin{split}
\iota_{z_1, z_2}z_{12}^{n}Y(v&,z_1)\Y(a, z_2)-\iota_{z_2, z_1}z_{12}^{n}\Y(a, z_2)Y(v,z_1)\\
&=\sum_{i\geq 0}\Y( v_{(n+i)}a, z_2)\partial^{(i)}_{z_2}\delta(z_1, z_2)\, .
\end{split}
\end{equation}
\end{definition}
 Note that a $V$-module $M$ defines an intertwiner $\binom{M}{V \quad  M}$. 

\begin{remark} We are using an equivalent expression of the Jacobi identity in \cite{FHL}. The formulation here will be useful to describe the $\lambda$-bracket. Additionally, the identity \eqref{b} is equivalent to ($n,m\in \ZZ$, $k\in \CC$)
\begin{equation}\label{borcherds2}
\begin{split}
\sum_{j\in \ZZ_+} &(-1)^{j}\binom{n}{j}\left(v_{(m+n-j)}a_{(k+j)} -(-1)^{n}a_{(n+k-j)} v_{(m+j)}\right) \\
& =\sum_{j\in \ZZ_+} \binom{m}{j}( v_{(n+j)}a )_{(m+k-j)} \,.
\end{split}
\end{equation}

\end{remark}


Let $U$ be a complex vector space. The \emph{formal Fourier transform} $F^{\lambda}_{z}:U(\!(z)\!)\rightarrow U[\lambda]$ is  defined linearly by its action on  monomials as follows
\begin{equation}\label{fousual}F^{\lambda}_{z}(z^{-n-1})=  \begin{cases}
                 \lambda^{(n)}:=\frac{\lambda^n}{n!}\, ,  & n\in \ZZ_{\geq 0}   \,, \\
                  0\, ,  &n\in \ZZ_{< 0} \, .
                \end{cases}   \end{equation}
For a vertex algebra $V$ the $\lambda$-bracket is given by ($u,v\in V$)
\begin{equation}
\label{lambdava}
[u_{\lambda}v]:=F^{\lambda}_{z}(Y(u,z)v)\in V[\lambda]\, .
\end{equation}   
This bracket together with the endomorphisms $T$ forms a \emph{Lie conformal algebra}, see \cite{K} for details, satisfying \emph{sesquilinearity}, \emph{skewsymmetry} and \emph{Jacobi identity}:  ($u,v,w\in V$)
\begin{enumerate}
\item[]  $[T u_{\lambda}v]=-\lambda [ u_{\lambda}v]$,    $\quad[ u_{\lambda}T v]=(\lambda+T) [ u_{\lambda}v]$$,$
\item[] $[u_{\lambda}v]=-[v_{-\lambda-T}u]$$,$
\item[] \label{jac3} $[u_{\lambda}[v_{\mu}w]]=[v_{\mu}[u_{\lambda}w]]+[[u_{\lambda}v]_{\lambda+\mu}w]$.       
\end{enumerate}

Also, we have a product $\cdot: V\otimes V\rightarrow V$ given by
\begin{equation}\label{prova}u\cdot v:=u_{(-1)}v\in V\, .\end{equation}
This product has a unit $\vac$, a differential $T$,  and it is \emph{quasicommutative} and  \emph{quasiassociative}: ($u,v,w\in V$)
\begin{enumerate}
\item[] $uv-vu=\int^{0}_{-T}[u_{\lambda}v]d\lambda  $,
\item[]  $(uv)w-v(uw)=\left(\int^{T}_{0}d\lambda v\right)[u_{\lambda}w]+\left(\int^{T}_{0}d\lambda u\right)[v_{\lambda}w] $.
\end{enumerate}

Finally, the $\lambda$-bracket and the product are related by the \emph{noncommutative Wick formula}: ($u,v,w\in V$)
\begin{enumerate}
\item[] $[u_{\lambda}vw]=v[u_{\lambda}w]+[u_{\lambda}v]w+\int_{0}^{\lambda} [[u_{\lambda}v]_{\mu}w]d\mu$.
\end{enumerate}
\begin{theorem}\cite{BK} A vertex algebra is given by quintuple $(V, \vac,T, [\cdot_{\lambda}\cdot ], \cdot)$ satisfying the properties 
\begin{enumerate}
\item $(V, T, [\cdot_{\lambda}\cdot ])$ is a Lie conformal algebra.
\item $(V, \vac,T,[\cdot_{\lambda}\cdot ], \cdot)$ quasicommutative, quasiassociative unital diff algebra
\item $(V, [\cdot_{\lambda}\cdot ], \cdot)$ satisfies the noncommutative Wick formula.
\end{enumerate}
\end{theorem}

Now, for a $V$-module $M$ the $\lambda$-bracket is given by ($u\in V, a\in M$)
\begin{equation}\label{lambdam}
[v_{\lambda}a]:=F^{\lambda}_{z}(Y(u,z)a)\in M[\lambda]\, .
\end{equation}     
This bracket together with the endomorphisms $T$ forms a \emph{Lie conformal module} ($u,v\in V, a\in M$)
\begin{enumerate}
\item[] $[T u_{\lambda}a]=-\lambda [ u_{\lambda}a]$,    $\quad[ u_{\lambda}T v]=(\lambda+T) [ u_{\lambda}v]$$,$
\item[] \label{jac3} $[u_{\lambda}[v_{\mu}a]]=[v_{\mu}[u_{\lambda}a]]+[[u_{\lambda}v]_{\lambda+\mu}a]$.            
\end{enumerate}

We have the product $\cdot: V\otimes M\rightarrow M$ given by
\[u\cdot m:=u_{(-1)}m\in M\, .\]
This product is a representation of representation of the quasicommutative, quasiassociative unital diff algebra \eqref{prova} i.e.  $\vac\cdot =\text{Id}_{M}$, $T$ is a differential and  ($u,v\in V, a\in M$)
\begin{enumerate}
\item[] $(uv)a-u(va)=\left(\int^{T}_{0}d\lambda v\right)[u_{\lambda}a]+\left(\int^{T}_{0}d\lambda u\right)[v_{\lambda}a] ,$
\item[] $u(va)-v(ua)=\int^{0}_{-T}[u_{\lambda}v]d\lambda a$.
\end{enumerate}
Finally, the $\lambda$-bracket and the product are related by 
\begin{enumerate}
\item[] $[u_{\lambda}va]=v[u_{\lambda}a]+[u_{\lambda}v]a+\int_{0}^{\lambda} [[u_{\lambda}v]_{\mu}a]d\mu$,
\item[]  $[u  v_{\lambda}  a ]=(e^{T\partial_{\lambda}}u)[v_{\lambda}a]+(e^{T\partial_{\lambda}}v)[u_{\lambda}a]+\int_{0}^{\lambda}  [u_{\mu}[v_{\lambda-\mu}a]]d\mu $.

\end{enumerate}

\begin{proposition} A $V$-module is given by triple $(M, [\cdot_{\lambda}\cdot ], \cdot)$ satisfying
\begin{enumerate}
\item $(M, [\cdot_{\lambda}\cdot ])$ is a Lie conformal module.
\item $(M,  \cdot)$ is a representation of the quasicommutative, quasiassociative unital diff algebra.
\item $(M, [\cdot_{\lambda}\cdot ], \cdot)$ satisfies the noncommutative Wick formulas.\\\
\end{enumerate}
\end{proposition}

\section{Lambda bracket formalisms}\label{s1}

Let $U$ be a complex vector space. Now the general \emph{formal Fourier transform} $F^{\lambda}_{z}:U[\![z]\!]z^{-\Gamma}\rightarrow U[\![\lambda^{-1}]\!]\lambda^{\Gamma}$ is defined linearly by its action on the monomials as in \eqref{fo}. 
By definition we have that
\begin{lemma}\label{lem0}The map $F^{\lambda}_{z}$ satisfies 
\begin{enumerate}
\item $F^{\lambda}_z$ restricted to $U(\!(z)\!)$ gives us the formal Fourier transform \eqref{fousual};
\item \label{ii)}We have that $(n+1)\lambda^{(n+1)}=\lambda^{(n)}\lambda$ for $n\in \CC$;
\item We have $\iota_{\mu,\lambda}(\lambda+\mu)^{(n)}=\sum_{k\geq 0} \lambda^{(k)}\mu^{(n-k)}$ for $n\in\CC$.
\end{enumerate}
\end{lemma}
\begin{proof} $(1)$ follows from $\frac{1}{\Gamma(n+1)}=\frac{1}{n!}$ if $n\in \ZZ_{\geq 0}$. And $\frac{1}{\Gamma(n+1)}=0$ if and only if $n\in \ZZ_{< 0}$. $(2)$ follows from $\frac{1}{\Gamma(n)}=\frac{n}{\Gamma(n+1)}$ for all $n\in \CC$. Finally $(3)$ follows from $\binom{n}{j}=\frac{\Gamma(n+1)}{j!\Gamma(n-j+1)}$ if $n\in \CC-\ZZ_{<0}$, and if $n\in \ZZ_{<0}$ the identity is trivial.
\end{proof}

We describe now some properties of the Fourier transform

\begin{lemma}\label{lem1} The map $F^{\lambda}_{z}$ satisfies 
\begin{enumerate}
\item $F^{\lambda}_{z}  z=\partial_{\lambda}F^{\lambda}_{z}$;
\item $F^{\lambda}_{z} \partial_{z}=-\lambda F^{\lambda}_{z}$;
\item $F^{\lambda}_{z} e^{zT}= F^{\lambda+T}_{z}$;
\item $F^{\lambda}_{z}\left(\iota_{z,w}(z-w)^{-n-1}\right)=e^{\lambda w}\lambda^{(n)}$, $n\in \CC$;
\item $F^{\mu}_{w}F^{\lambda}_{z}\left(w^{-m-1}\iota_{z,w}(z-w)^{-n-1}\right)=\lambda^{(n)}\iota_{\mu,\lambda}(\lambda+\mu)^{(m)}$, $n,m\in \CC$.
\end{enumerate}
\end{lemma}
In $(3)$ we assume that $\lambda+T$ is expanded  in positive powers of the endomorphisms $T$  i.e. for $n\in \CC$ we have $(\lambda+T)^{(n)}=\sum_{k\geq 0}\lambda^{(n-k)}T^{(k)}$.
\begin{proof} (1) follows from $F^{\lambda}_{z}(zz^{-n-1})=\lambda^{(n-1)}=\partial_{\lambda}\lambda^{(n)}$.  (2) follows from
 $F^{\lambda}_{z}\left( \partial_z z^{-n-1}\right)=-(n+1)\lambda^{(n+1)}={-\lambda}\lambda^{(n)}$.  (3) follows from the identity $F^{\lambda}_{z}e^{zT}=e^{T\partial_{\lambda}}F^{\lambda}_{z}=F^{\lambda+T}_{z}$.
 (4) for $n\in \ZZ_{<0}$ the identity is obvious, for $n\in \CC- \ZZ_{<0}$ we have 
  \begin{align*} 
F^{\lambda}_{z}\left(\iota_{z,w}(z-w)^{-n-1}  \right)&=F^{\lambda}_{z}\sum_{j\geq 0 }(-1)^{j}\binom{-n-1}{j}z^{-j-n-1}{w^{j}}\\
&=\sum_{j\geq 0}\binom{n+j}{j}{w^{j}}\lambda^{(n+j)}=\lambda^{(n)}e^{\lambda w}\, .
\end{align*} 
Finally, (5) follows using $(4)$ 
\begin{align*}
F^{\mu}_{w}F^{\lambda}_{z}&\left(w^{-m-1}\iota_{z,w}{(z-w)^{-n-1}} \right)=F^{\mu}_{w}\left(w^{-m-1}\lambda^{(n)}e^{\lambda w}\right)\\
&=\sum_{k\geq 0} \lambda^{(n)}\lambda^{(k)}\mu^{(m-k)}=\lambda^{(n)}\iota_{\mu,\lambda}(\lambda+\mu)^{(m)}\, .
\end{align*}
 \end{proof}

Now, as we mentioned in the introduction, we define the following $\lambda$-bracket\footnote{We choose the notation $[a_{\lambda}b]$ instead of $a_{\lambda}b$ because in future works we will consider algebraic relations generalizing the notion of vertex algebras, \cite{DL, BK2, BV}.} for  intertwiners
\begin{equation*}\label{lambdaint}
[a_{\lambda}b]:=F^{\lambda}_{z}(\Y(a, z)b)\in M_3[\![\lambda^{-1}]\!]\lambda^{\Gamma}\, .
\end{equation*}


We have the following properties state on the introduction

\begin{proposition}\label{pro3.2} The $\lambda$-bracket for intertwiners satisfies $\textnormal{i)}$ and $\textnormal{ii)}$

\end{proposition}
\begin{proof} i) follows from Definition \ref{defint} $(i)$, Lemma \ref{lem1}(2) and because the operator $T$ defines a derivation. 
 ii) is equivalent to identity \eqref{borcherds2} for $n=0$, $m\in \ZZ_{\geq 0}$ and $k\in \CC-\ZZ_{<0}$  
\begin{equation}\label{eq4.1}
\begin{split}
v_{(m)}a_{(k)}b -a_{(k)} v_{(m)}b =\sum_{j\in \ZZ_+} \binom{m}{j}( v_{(j)}a )_{(m+k-j)} b\,.
\end{split}
\end{equation}





\end{proof}

 Now the product \eqref{eqpro} satisfies  by Definition \ref{defint} that
\[(\frac{1}{k!}T^ka)\cdot b = a_{(-1-k)}b\,\quad \text{and}\quad T(a\cdot b)=Ta\cdot b+a\cdot Tb\, .\]


And the product satisfies the properties state on the introduction

\begin{proposition}\label{pro3.3} The product of intertwiners satisfies $\textnormal{iii)}$ and $\textnormal{iv)}$ 
\end{proposition}
\begin{proof}  $\textnormal{iii)}$  is equivalent to \eqref{borcherds2} for $m=0$, $n=k=-1 $ 
 \begin{equation}\label{eq3.3ref}
\begin{split}
\sum_{j\in \ZZ_+} &\left(v_{(-1-j)}a_{(-1+j)}b +a_{(-2-j)} v_{(j)}b\right)  =( v_{(-1)}a )_{(-1)}b \,.
\end{split}
\end{equation}
\vspace{-0,3cm}
Note that $\sum_{j\geq 0}{ a}_{(-j-2)}v_{(j)}b=\left(\int^{T}_{0}d\lambda a\right)[v_{\lambda}b]$ and 
\[
\sum_{j\in \ZZ_+} v_{(-1-j)}a_{(-1+j)}b =v  ( a  b)+\sum_{m\geq 0}v_{(-m-2)}{ a}_{(m)}b=v  ( a  b)+\left(\int^{T}_{0}d\lambda v\right)[{ a}_{\lambda}b]^0\, .
\]


$\textnormal{iv)}$  is equivalent to \eqref{borcherds2} for $n=0$, $m=k=-1 $ 
\begin{equation}\label{eq3.4ref}
\begin{split}
v_{(-1)}a_{(-1)}b -a_{(-1)} v_{(-1)}b =\sum_{j\in \ZZ_{\geq 0}} (-1)^{j}( v_{(j)}a )_{(-2-j)} b\,.
\end{split}
\end{equation}
Note that $-\frac{1}{(j+1)!}(-T)^{j+1}( v_{(j)}a )_{(-1)} b=(-1)^{j}( v_{(j)}a )_{(-2-j)} b$.

\end{proof}

Now, we introduce the liner map $\int^{\lambda}_{0}:U[\![\lambda^{-1}]\!]\lambda^{\Gamma} \rightarrow U[\![\lambda^{-1}]\!]\lambda^{\Gamma}$ defined linearly by its action on the monomials as follows: For $k\in \CC-\ZZ_{<0}$
\begin{equation}\label{defintegral} \int^{\lambda}_{0}\mu^{(k)}d\mu:=\lambda^{(k+1)}\, .\end{equation} 
For $k\in \ZZ_{<0}$, it is defined to be zero.
\begin{lemma}\label{lem3.4} For, $k\in \CC-\ZZ_{<0}$ and $i\in \ZZ_{\geq 0}$ we have 
\[\lambda^{(k+i+1)}=\int^{\lambda}_{0}\mu^{(k)}(\lambda-\mu)^{(i)}d\mu \, .\]
\end{lemma}
\begin{proof} From Chu-Vandermonde identity $1=\sum_{j\geq 0}^{i}\binom{-k-1}{j}\binom{k+i+1}{i-j}$. Then  
\begin{align*}
\lambda^{(k+i+1)}&=\sum_{j\geq 0}^{i}(-1)^{j}\binom{k+j}{j}\binom{k+i+1}{i-j}\lambda^{(k+i+1)}\\
&=\sum_{j\geq 0}^{i}(-1)^{j}\binom{k+j}{j}\lambda^{(i-j)}\int^{\lambda}_{0}\mu^{(k+j)}d\mu\\
&=\int^{\lambda}_{0}\sum_{j\geq 0}^{i}(-1)^{j}\mu^{(k)}\mu^{(j)}\lambda^{(i-j)}d\mu\, .
\end{align*} 
\end{proof}

Now the $\lambda$-bracket and the product satisfy the following properties

\begin{proposition}\label{pro3.4} The $\lambda$-bracket and the product for intertwiners satisfy $\textnormal{v)}$,  $\textnormal{vi)}$ and $\textnormal{vii)}$
\end{proposition}

\begin{proof} 
$\textnormal{v)}$ is equivalent to identity \eqref{borcherds2} for $n=0$, $m\in \ZZ_{\geq 0}$ and $k=-1$  
\begin{equation}\label{eq3.4}
\begin{split}
v_{(m)}a_{(-1)}b -a_{(-1)} v_{(m)}b =\sum_{j\in \ZZ_{\geq 0}} \binom{m}{j}( v_{(j)}a )_{(m-1-j)} b\,.
\end{split}
\end{equation}

$\textnormal{vi)}$ is equivalent to the identity \eqref{borcherds2} for $m=0$, $n=-1$, $k\in\CC- \ZZ_{< 0}$
\begin{equation}\label{eq3.5}
\sum_{j\in \ZZ_{\geq 0}} \left(v_{(-1-j)}a_{(k+j)}b +a_{(-1+k-j)} v_{(j)}b\right)  =( v_{(-1)}a )_{(k)}b\,.
\end{equation}
Note that 
\[\sum_{k\in\CC- \ZZ_{<0}}  \sum_{j\in \ZZ_{\geq 0}} \lambda^{(k)}a_{(-1+k-j)} v_{(j)}b=\left(\sum_{k\in \ZZ_{\geq 0}}+\sum_{k\in \CC-\ZZ}\right) \sum_{j\in \ZZ_{\geq 0}} \lambda^{(k)}a_{(-1+k-j)} v_{(j)}b\, , \]
where 
\begin{align*}
\sum_{k\in \ZZ_{\geq 0}}\sum_{j\in \ZZ_{\geq 0}} \lambda^{(k)}a_{(-1+k-j)} v_{(j)}b&=\sum_{k\in \ZZ_{\geq 0}}\left(\sum_{j\geq k} +\sum_{k>j\geq 0}\right)\lambda^{(k)}a_{(-1+k-j)} v_{(j)}b\\
&=(e^{T\partial_{\lambda}}a)[v_{\lambda}b]+\int_{0}^{\lambda}  [a_{\mu}[v_{\lambda-\mu}b]]d\mu\,, 
\end{align*}
 on the second term on the right-hand side we use Lemma \ref{lem3.4}. Also, by Lemma  \ref{lem3.4} we have 
\[\sum_{k\in \CC-\ZZ}\sum_{j\in \ZZ_{\geq 0}} \lambda^{(k)}a_{(-1+k-j)} v_{(j)}b=\int_{0}^{\lambda}  [a_{\mu}[v_{\lambda-\mu}b]]d\mu\, .\]

$\textnormal{vii)}$ is equivalent to the identity \eqref{borcherds2} for $n=0$, $m=-1$, $k\in \CC-\ZZ_{<0}$
\begin{equation}\label{eq3.6b}v_{(-1)}a_{(k)}b -a_{(k)} v_{(-1)}b =\sum_{i\in \ZZ_{\geq 0}}(-1)^i( v_{(i)}a )_{(-1+k-i)} b\,.
\end{equation}
Now $[a_{\lambda}v]=-[v_{-\lambda-T}a]$ is equivalent to $v_{(i)}a=-\sum_{j\geq 0}(-1)^{i+j}\frac{1}{j!}T^{j}(a_{(i+j)}v)$ for $i\geq 0$. Hence   
\begin{align*}
&\sum_{i\geq 0}(-1)^{i}(v_{(i)}a)_{(-1+k-i)}b=
-\sum_{r=i+j\geq 0}\binom{k}{r}(a_{(r)}v)_{(-1+k-r)}b ,
\end{align*}
using that $(T^{(k)} a)_{(n)}=(-1)^k \binom{n}{k}a_{(n-k)}$ and that  ${ \sum _{j=0}^{l}{\binom {n+j}{j}}={\binom {n+l+1}{l}}}$ for $l\in \ZZ_{\geq 0}$, $n\in \CC$.
\end{proof}

We have the following Proposition, see \cite[Proposition 4.8]{K}.

\begin{proposition} \label{p3.7}The identity \eqref{b} is equivalent to the following two identities
\begin{align*}
&\Y(v_{(-1)} a, z)b=Y_{+}(v,z)\Y(a,z)b+\Y(a,z) Y_{-}(v,z)b\, ,\\
&Y(v,z_1)\Y(a, z_2)b-\Y(a, z_2)Y(v,z_1)b=\sum_{i\geq 0}\Y(v_{(i)}a, z_2)b\partial^{(i)}_{z_2}\delta(z_1, z_2)\, ,
\end{align*}
where $Y_{+}(v,z):=\sum_{n<0} v_{(n)}z^{-n-1}$ and $Y_{-}(v, z):=\sum_{n\geq 0} v_{(n)}z^{-n-1}$.
\end{proposition}


The proof follows  the same steps  in \cite[Theorem 4.8]{K}. 
In the following theorem we assume that for all $a\in M_1, b\in M_2$
\begin{equation}\label{eq3.9}[a_{\lambda}b]=\sum_{n\in \CC-\ZZ_{< 0}} \lambda^{(n)} c_{(n)}\, ,\end{equation}  
where $c_{(n)}\in M_3$. The proof of the theorem below is similar to \cite{BK}.

\begin{theorem}\label{theorem3.8}Let $M_1 , M_2, M_3$ be $V$-modules. An intertwiner of type $\binom{M_3}{M_1\quad M_2}$ is equivalently defined by a $\lambda$-bracket satisfying \eqref{eq3.9} and a product
\[[\cdot_{\lambda}\cdot]:M_1\otimes M_2\rightarrow M_3[\![\lambda^{-1}]\!]\lambda^{\Gamma}\, , \qquad \cdot:M_1\otimes M_2\rightarrow M_3\] such that
\begin{enumerate}
\item The $\lambda$-bracket $[\cdot_{\lambda}\cdot]$ satisfies $\textnormal{i)}$, $\textnormal{ii)}$.
\item The product $\cdot$ has derivative $T$ and satisfies $\textnormal{iii)}$, $\textnormal{iv)}$.
\item $[\cdot_{\lambda}\cdot]$ and $\cdot$  satisfies $\textnormal{v)}$, $\textnormal{vi)}$, $\textnormal{vii)}$.
\end{enumerate}

\end{theorem}

\begin{proof} For $a\in M_1$, $b\in M_2$  we define, see \eqref{eq3.9}, 
\[a_{(n)}b:= c_{(n)} ,\quad n\in \CC-\ZZ_{< 0} \quad\text{ and } \quad a_{(-1-n)}b:=(\frac{1}{n!}T^{n}a)\cdot b, \quad  n\in \ZZ_{\geq 0}\, .  \] 
And, we define $\Y(a,z)b:=\sum_{n\in \CC} a_{(n)}b\, z^{-n-1}\in M_3[\![z]\!]z^{-\Gamma}$. Then, from $\textnormal{ii)}$ and $T$ being a differential of the product we obtain
\begin{equation}\label{trans}[T,\Y(a,z)]=\Y(T a,z)b=D_{z}\Y(a,z)\, .\end{equation}

We have the identity
\begin{equation}\label{eq3.12}
\begin{split}
v_{(m)}a_{(k)}b -a_{(k)} v_{(m)}b =\sum_{j\in \ZZ_{\geq 0}} \binom{m}{j}( v_{(j)}a )_{(m+k-j)} b\,,
\end{split}
\end{equation}
for $m\in \ZZ_{\geq 0}$, $k\in \CC-\ZZ_{<0}$ from $\textnormal{ii)}$, see \eqref{eq4.1}; for $m\in \ZZ_{\geq 0}$,  $k=-1$ from $\textnormal{v)}$, see \eqref{eq3.4}; for $m=-1$, $k\in \CC-\ZZ_{<0}$ from $\textnormal{vii)}$, see \eqref{eq3.6b}; and for $m=-1$, $k=-1$ from $\textnormal{iv)}$, see \eqref{eq3.4ref}. Then, using translation covariance we obtain \eqref{eq3.12} for  $m\in \ZZ$, $k\in \CC$. 

And, we have the identity
\begin{equation}\label{eq4.3}
\sum_{j\in \ZZ_{\geq 0}} \left(v_{(-1-j)}a_{(k+j)}b +a_{(-1+k-j)} v_{(j)}b\right)  =( v_{(-1)}a )_{(k)}b\,.
\end{equation}
For $k\in \CC- \ZZ_{< 0}$ from $\textnormal{vi)}$, see \eqref{eq3.5}. And for $k=-1$ from $\textnormal{iii)}$, see \eqref{eq3.3ref}. Then using translation covariance we obtain \eqref{eq4.3} for $k\in \CC$.

Finally, \eqref{eq3.12} and \eqref{eq4.3} are equivalent to the identities in Proposition \ref{p3.7}. Hence, we obtain \eqref{b}.
\end{proof}




\section{Applications}\label{s4}
\subsection{Virasoro algebra} 

The Virasoro algebra is the Lie algebra $\mathfrak{vir}=\left(\bigoplus_{n\in \ZZ}\CC L_{n}\right)\oplus \CC C$ with commutations relations 
\[[L_n, L_m]=(n-m)L_{m+n}+\frac{n^3-n}{12}C\delta_{n,-m}, \qquad [L_n , C]=0\, .\]

Now, we have  the space
\[\text{Vir}^c:=U(\text{Vir})\otimes_{U(\bigoplus_{n\geq -1}\CC L_{n}\oplus \CC C )}\CC_c\]
where $C$ acts by $c$ and  $L_{n}$ acs by zero for $n\geq -1$. $\text{Vir}^c$ is the universal Virasoro vertex algebra, $\vac=1\otimes 1$, $T=L_{-1}$ and $Y:\text{Vir}^c\rightarrow \Hom (\text{Vir}^c, \text{Vir}^c((z)))$
 is given by
\[Y((L_{-n_1-2})\cdots (L_{-n_r-2})\vac_{}, z):=:\partial^{(n_1)}_{z}L(z)\cdots \partial^{(n_r)}_{z}L(z):\]
where $L(z)=\sum_{n\in \ZZ}L_n z^{-n-2}$ and $:\, :$ denotes the normally ordered product, \cite[Sec 3.1]{K}. 

For $L=(L_{-2})\vac\in \text{Vir}^c$ the $\lambda$-bracket, see \eqref{lambdava}, is given by 
\begin{equation}\label{eqvir}
[L_{\lambda}L]=(2\lambda+\partial)L+\frac{c}{12}\lambda^{3}\, .
\end{equation}

Let $M$ be a $\text{Vir}^c$-module, we say that $a\in M$ is a \emph{primary vector}  if $L_{n}a=0$ for $n>0$ and $L_{0}a=h_a a$. Such vectors are also known as singular, or null, vectors. Equivalently using the $\lambda$-bracket, see \eqref{lambdam}, we have
\[[{L}_{\lambda}a]=(T+h_a \lambda)a\, .\]

Now, we consider intertwiners of type $\binom{M_3}{M_{1}\quad M_{2} }$.  By definition  \ref{defint}, it exists $\{m_1, \cdots , m_d\}\in \Gamma\subset \CC$ such that for $a\in M_1$, $b\in M_2$
\begin{equation}\label{eq4.2int}
\Y(a, z)b=\sum_{n\in \Gamma}a_{(n)}bz^{-n-1}=\sum_{i}\sum_{n\geq 0}{a}_{(m_i-n)}bz^{-m_i+n-1}
\end{equation}

The $\lambda$-bracket, see \eqref{lambdaint},  gives us 
\begin{equation}\label{eq5.1}
[{a}_{\lambda}b]=\sum_{i}\sum_{n\geq 0}\lambda^{(m_i-n)}{a}_{(m_i-n)}b\, .
\end{equation}

We have the following Lemma 



\begin{lemma}\label{pro5.1} Let $a\in M_1, b\in M_2$ be primary vectors. Then 
\[[L_{\lambda}{a}_{(m_i)}b]=(\lambda h_c+\partial){a}_{(m_i)}b\]
 where $m_i=h_a+h_b -1-h_{c}$ i.e. ${a}_{(m_i)}b$ is a primary vector. 
\end{lemma}
\begin{proof} From Jacobi identity \textnormal{ii)} we have that 
\begin{equation}\label{eq5.2}
\begin{split}
[L_{\lambda}[a_{\mu}b]]&=\iota_{\mu, \lambda}[[L_{\lambda}a]_{\lambda+\mu}b]+[a_{\mu}[L_{\lambda}b]]\\
&=(h_a \lambda -(\lambda+\mu))\iota_{\mu, \lambda}[a_{\lambda+\mu}b]+(h_b \lambda+\mu+T)[a_{\mu}b]\, .
\end{split}
\end{equation}
The coefficient of $\mu^{(m_i)}$ on both sides gives us the identity. 
If $m_i< 0$ then $a_{(m_i)}b=(T^{(-m_i-1)}a)b$ and for $l:=-m_i-1$ 
\begin{align*}
[L_{\lambda}(T^{(l)}a)b]&=((\lambda+T)^{(l)}(\lambda h_a+T)a)b+T^{(l)}a(h_b\lambda+T)b\\
&=(h_a+h_b+l)\lambda (T^{(l)}a)b+T((T^{(l)}a)b)\, .
\end{align*}
\end{proof}


Now, a Verma module, for $h,c\in \CC$ is given by
\[M^{c}_{h}:=U(\mathfrak{vir})\otimes_{U(\bigoplus_{n\geq 0}\CC L_{n}\oplus \CC C )}m\]
where $C m=cm$, $L_0 m=hm$ and $L_{n}m=0$, $n\geq 1$. The vector $m$ is the highest weight vector of $M^{c}_{h}$. Let $L^{c}_{h}$ the quotient of the Verma modules $M^{c}_{h}$ by its unique non-trivial maximal submodule, see \cite{KRR}. We consider $c$ and $h$ as follows
\[c_{p,q}=1-\frac{6(p-q)^2}{pq}\]
\[h_{k,l}=\frac{(lp-kq)^2-(p-q)^2}{4pq}\]
where $(p,q)=1$, $0<k<p$, $0<l<q$. 

Now, it follows from a direct calculation that 
$\bigl(L_{-2}-\frac{3}{2(2h_{1,2}+1)}L_{-1}^{2}\bigr)b$
is a singular vector for $M^{c}_{h_{1,2}}$, where $b$ denotes the higuest weight vector. Using this we obtain the following result from \cite{BPZ}


\begin{proposition}\label{proBPZ} The intertwiners of type  $\binom{L^{c}_{h}}{L^c_{h_{k,l}}\quad L^c_{h_{1,2}} }$ are trivial unless 
\begin{equation}\label{bpz}0=-\kappa+h_{k,l} -\frac{3}{2(2h_{1,2}+1)}(\kappa)(\kappa-1)\end{equation}
where $\kappa=h-h_{k,l} -h_{1,2}$.
\end{proposition}
\begin{proof} Let $a\in L^c_{h_{k,l}}$, $b\in L^c_{h_{1,2}}$ be the respective highest weight vectors.  From Lemma \ref{pro5.1}, we have that $a_{(m)}b$ is the highest weigh vector of $L^{c}_{h}$ for $m=-h+h_{1,2}+h_{k,l}-1$.

From skew-symmetry $[a_{\lambda}L]=-[L_{-\lambda-T}a]
=h_{k,l}\lambda a+(h_{k,l}-1)T a$. 
Let $\beta=\frac{3}{2(2h_{1,2}+1)}$. Then, we have 
\begin{align*}
0&=[{a}_{\lambda}\left(L-\beta T^{2}\right)b]=[{a}_{\lambda}Lb]-[{a}_{\lambda}\beta T^{2}b]\\
&=L[{a}_{\lambda}b]+[a_{\lambda}L]b+\int_0^{\lambda} (h_{k,l}\lambda-(h_{k,l}-1)\mu)[{a}_{\mu}b]d\mu -\beta(\lambda+ T)^{2}[{a}_{\lambda}b]\, .
\end{align*}
In particular, the coefficient of $\lambda^{(m+2)}$ for each term on the right-hand side above gives us
\begin{align*}
0=
h_{k,l}+(m+1)-\frac{3}{2(2h_{1,2}+1)}(m+1)(m +2)\, .
\end{align*}
Finally, we replace $\kappa=h-h_{k,l}-h_{1,2}= -1-m$. 

For $m<0$, we have using \textnormal{iv)} for $-m-1=l$ and $0=(T^{(l-2)}a)(L-\beta T^{2}b)$ the same identity is obtained. 

\end{proof}

The solutions of the second order equation \eqref{bpz} gives us $h=h_{k,l-1}$ or $h=h_{k,l+1}$. These solutions led the authors in \cite{BPZ} to formulate the fusion rules of the minimal models.




Now, we generalize the identity \eqref{bpz}. First, we prove the following Lemma for interwiners  of type $\binom{L^{c}_{h}}{L^c_{h_{k,l}}\quad L^c_{h_{r,s}} }$. 

\begin{lemma}\label{lem4.3} Let $a\in L^c_{h_{k,l}}$ be a highest weight vector and $b\in L^c_{h_{r,s}}$an arbitrary vector. For $j\geq -1$ 
\begin{equation*}
\begin{split}
[a_{\lambda}L_{-j-2}b]=\lambda^{(m+j+2)}\left((j+1)h_{k,l}+(m+1)\right)a_{(m)}b+\sum_{n\geq 1}\lambda^{(m+j+1-n)}\cdots 
\end{split}
\end{equation*}
where $m=-h+h_{r,s}+h_{k,l}-1$.
\end{lemma}
\begin{proof}
For $j\geq 0$ we have that
\[[a_{\lambda}L_{-j-2}b]=T^{(j)}L[a_{\lambda}b]+[a_{\lambda}T^{(j)}L]b+\int^{\lambda}_{0}[[a_{\lambda}T^{(j)}L]_{\mu}b]d\mu\]
for $a\in L^c_{h_{k,l}}$, $b\in L^c_{h_{r,s}}$. 
 Then because $a\in L^c_{h_{k,l}} $ is primary we have using \textnormal{i)} and Lemma \ref{lem3.4} on the integral term above that 
\[[a_{\lambda}L_{-j-2}b]=\lambda^{(m+j+2)}\left((j+1)h_{k,l}+(m+1)\right)a_{(m)}b+\lambda^{(m+j+1)}\cdots \]
Where,  $m=-h+h_{r,s}+h_{k,l}-1$ from Lemma \ref{pro5.1}. Finally, using \textnormal{i)} we have that the identity is also satisfied for $j=-1$.
\end{proof}

Let $b\in L^c_{h_{r,s}}$ be the highest weight vector. From Kac determinant formula, see \cite{KRR}, and the submodule structure \cite{FF}, we have that there is a singular vector of $M^{c}_{r,s}$ given by
\[\sigma_{r,s}(t)b=\sum_{\substack{j_1\geq \cdots \geq j_l\geq 1 \\ j_1+\cdots +j_l=rs}}\rho_{j_1, \cdots , j_l}(t)L_{-j_1}\cdots L_{-j_l}b\]
where $t=-q/p$, and $\rho_{j_1, \cdots , j_l}(t)\in \CC$. 

Recall that the Witt algebra is $\bigoplus_{n\in \ZZ}\CC l_{n}$ where $[l_n, l_m]=(n-m)l_{n+m}$.
 




\begin{theorem}\label{th4.3}\cite{FF, FF1} Let $\F_{\lambda, \mu}$ be a Witt algebra module with base $f_j$ ($j\in \ZZ$) and  action given by $l_{-i}f_j=(\mu+j-\lambda (i+1))f_{j+i}$. Define the function $\rho_{r,s}(\lambda, \mu, t)$ by the formula $\sigma_{r,s}(t)	f_{0}=\rho_{r,s}(\lambda, \mu, t)f_{rs}$. Then 
\begin{align*}
&\rho_{r,s}(\lambda, \mu, t)^{2}=\prod_{\substack{0\leq u<r \\ 0\leq v<s}} R_{r,s,u,v}(\lambda, \mu, t)\, , 
\end{align*}
\vspace{-0,4cm}
{\small \begin{align*}
&R_{r,s,u,v}(\lambda, \mu, t):=(\mu-2\lambda)^2+(\mu-2\lambda)(rs-(r-1-2u)(s-1-2v)-1)\\
&+(\mu-2\lambda)((2u(r-1-u)+r-1)t+(2v(s-1-v)+s-1)t^{-1})\\
&-\lambda((r-1-2u)^2 t+2(r-1-2u)(s-1-2v)+(s-1-2v)^2t^{-1})\\
&+(ut+v)((u+1)t+(v+1))((r-u)t+(s-v))((r-1-u)t+(s-1-v))t^{-2}\, .
\end{align*}}
\end{theorem} 

Therefore, we have from Lemma \ref{lem4.3} and Theorem \ref{th4.3} on intertwiners of type $\binom{L^{c}_{h}}{L^c_{h_{k,l}}\quad L^c_{h_{r,s}} }$ that
\begin{equation*}
\begin{split}
0=[a_{\lambda}\sigma_{r,s}(t)b]&=\lambda^{(m+rs)}\rho_{r,s}(-h_{k,l}, h_{r,s}-h-h_{k,l}, -{q}/{p})a_{(m)}b+\lambda^{(m+rs-1)}\cdots 
\end{split}
\end{equation*}
Hence, the generalization of identity \eqref{bpz} is given by
\[\rho_{r,s}(-h_{k,l}, h_{r,s}-h-h_{k,l}, -{q}/{p})=0\,. \]

Finally, we note that  
\begin{align*}
R_{r,s,u,v}(-h_{k,l},& h_{r,s}-h-h_{k,l}, -{q}/{p})=h_{r-2u, s-2v}h_{r-2(u+1), s-2(v+1)}\\
&-\frac{((r-2u-1)q-(s-2v-1)p)^2}{2pq}(h_{k,l}+h)+(h_{k,l}-h)^{2}
\end{align*}
with solutions given by $h=h_{k-r+1+2u,l-s+1+2v}$ and $h=h_{k+r-1-2u,l+s-1-2v}$. 

\subsection{Affine Kac-Moody algebra}

Let $\lieg$ be a finite dimensional complex simple Lie algebra with Cartan subalgebra $\lieh$. We consider the invariant form $(a, b)=\frac{1}{2h^{\vee}}\text{Tr}\left(\ad(a)\ad(b)\right)$ for $a,b\in \lieg$, and $h^{\vee}$ the dual Coxeter number. 
The Casimir operator $C:=\sum x_{i}x^{i}$, where $x_i, x^i$ are dual basis. Recall that on a highest weight module $E$ of highest weight $\alpha\in \lieh^{*}$, we have that 
\[C|_{E}=(\alpha, \alpha+2\rho)\text{Id}_{E}.\]

The affine Kac-Moody Lie algebra associated to $\lieg$ is the Lie algebra $\hat{\mathfrak{g}}=\lieg[t,t^{-1}]\oplus \CC K$ with commutations relations 
\[[xt^n, yt^m]=[x,y]t^{m+n}+n(x,y)\delta_{n,-m}K, \qquad [\hat{\lieg} , K]=0\, .\]


Now, we have  the space
\[V^{k}(\lieg):=U(\hat{\mathfrak{g}})\otimes_{U( \lieg[t]\oplus \CC K )}\CC_{k}\, , \]
where $K$ acts by $k$ and $\lieg[t]t$ acts by zero.
 $V^k(\lieg)$ is the universal affine vertex algebra, where $\vac=1\otimes 1$, $Y:V^k(\lieg)\rightarrow \Hom (V^k(\lieg), V^k(\lieg)((z)))$ given by
\[Y((x_1t^{-n_1-1})\cdots (x_rt^{-n_r-1})\vac_{}, z):=:\partial^{(n_1)}_{z}x_1(z)\cdots \partial^{(n_r)}_{z}x_r(z):\]
where $x(z)=\sum_{n\in \ZZ}(xt^n) z^{-n-1}$ for $x\in \lieg$ and $:\, :$ denotes the normally ordered product. 

For $J_{x}=(xt^{-1})\vac\in V^{k}(\lieg)$ the $\lambda$-bracket, see \eqref{lambdava}, is given by 
\begin{equation}\label{eqlie}[{J_{x}}_{\lambda} J_{y}]=J_{[x,y]}+k(x,y)\lambda\, . 
\end{equation}
Let $M$ be a $V^k(\lieg)$-module, and $a\in M$ such that  if $xt^{n}a=0$ for $n>0$ and we denote $(xt^0) a=xa$. Using the $\lambda$-bracket, see \eqref{lambdam}, we have equivalently
\begin{equation}\label{eq4sin}[{J_{x}}_{\lambda}a]=(xa)\, .\end{equation}
We describe the intertwiners of $V^k(\lieg)$ as \eqref{eq4.2int} and \eqref{eq5.1}
\begin{lemma}  
\label{pro5.6} Let $a\in M_{1}$, $b\in M_{2}$ satisfying \eqref{eq4sin}. Then 
\[[{J_{x}}_{\lambda}{a}_{(m_i)}b]={(x a)}_{(m_i)}b+a_{(m_i)}{(x b)}\, , \]
i.e. ${a}_{(m_i)}b$ satisfies \eqref{eq4sin} with $(xt^0) a_{(m_i)}b:={(x a)}_{(m_i)}b+a_{(m_i)}{(x b)}$. 
\end{lemma}



\begin{proof} It follows from \textnormal{ii)} that 
\begin{align*}
&[{J_{g}}_{\lambda}[a_{\mu}b]]=\iota_{\mu, \lambda}[{ga}_{\lambda+\mu}b]+[a_{\mu}g b]
\end{align*}
The coefficient of $\mu^{(m_i)}$ on both sides gives us the identity. 
If $m_i< 0$ then $a_{(m_i)}b=(T^{(-m_i-1)}a)b$ and for $l:=-m_i-1$ 
\[[{J_{x}}_{\lambda}(T^{(l)}a)b]=\left((\lambda+T)^{(l)}(xa)\right)b+T^{(l)}a(xb)= (T^{(l)}xa)b+(T^{(l)}a)xb\, .\]

\end{proof}

For a highest weight $\lieg$-module $E$ of highest weight $\alpha\in \lieh^{*}$  a \emph{Weyl module} is a $\hat{\lieg}$-module $M^{k}_{E}$ satisfying that $K$ acts by $k$ and  
\[U(\hat{\lieg})E=M^{k}_{E}, \quad (\lieg[t]t) E=0 , \quad  (x t^0) E=x  E  \, .\]

From the Segal-Sugawara construction, 
we have for $k\neq -h^{\vee}$ that 
\[L:= \frac{1}{2(k+h^{\vee})}\sum_i{J_{x_i}J_{x^i}}\in V^{k}(\lieg)\]
 satisfies \eqref{eqvir} with $c=\frac{k\dim \lieg}{(k+h^{\vee})}$. 

The next lemma is known in the literature, see \cite{KZ} 
\begin{lemma}\label{lem5.4} Let $M^{k}_{E}$ as above. For $k\neq -h^{\vee}$, we have for $a\in E$ that 
\[[L_{\lambda}a]=\frac{1}{(k+h^{\vee})}J_{g_i} (g^ia)+\lambda h_{a}a\, \quad\text{where}\quad h_{a}:=\frac{(\alpha,\alpha+2\rho) }{2(k+h^{\vee})} \]
\end{lemma}
\begin{proof}
 From noncommutative Wick theorem we have
\begin{align*}
[a_{\lambda}J_{g_i}J_{g^i}]&=-{(g_i a)}J_{g^i}-J_{g_i}{(g^i a)}-\int^{\lambda}_{0}[{{g_i a}}_{\mu}J_{g^i}]d\mu\\
&=-{g_i a}J_{g^i}-J_{g_i}{g^i a}+\lambda (\alpha,\alpha+2\rho) a
\end{align*}
where we used the skew-symmetry $[a_{\lambda}J_g]=-ga$. Additionally, from quasicommutativity we have that 
\[{g_i a}J_{g^i}-J_{g_i}{g^i a}=\int^{0}_{-T}[{g_i a}_{\lambda}J_{g^i}]=-(\alpha,\alpha+2\rho)T a\,. \]
Then $[{J_{g_i}J_{g^i}}_{\lambda}a]
=2J_{g_i}{(g^i a)}+\lambda(\alpha,\alpha+2\rho){ a}$.
\end{proof}

We assume that $E$ is $\lieg$-irreducible finite dimensional with highest weight $\alpha$. Let $L^{k}_{\alpha}$ the quotient of a modules $M^{k}_{E}$ by its unique non-trivial maximal submodule, see \cite{K1}. Now, we assume  
\[k\in \ZZ_{\geq 0}\,,  \qquad (\alpha, \theta)\leq k\, .\]
It follows from \cite[Lem 10.1]{K1} that $\left(e_{\theta}t^{-1}\right)^{k-(\alpha, \theta)+1}b$ is a singular vector for $M^{k}_{E}$, where $b\in M^{k}_{E}$ denotes its higuest weight vector.

Now, we use the notation $E_{\alpha_1}\otimes E_{\alpha_2}=\bigoplus E_{\alpha_i}$ the decomposition on irreducibles. Then  obtain the following identity from \cite{GW}.


\begin{proposition}\label{proBPZ} The intertwiners of type  $\binom{L^{k}_{\alpha}}{L^k_{\alpha_1}\quad L^k_{\alpha_2} }$ are trivial unless $E_{\alpha}=E_{\alpha_i}$ for some $i$ and $\forall a\in E_1$ and $b$ the highest weight vector of $E_2$
\begin{equation}\label{gw}(e^{k-(\alpha_2,\theta)+1}_{\theta}a)\otimes b|_{E_{\alpha}}\, =0\, .\end{equation}
\end{proposition}
\begin{proof} Let $ E_1\subset  L^k_{\alpha_1}$, $ E_2\subset  L^k_{\alpha_2}$.  From Lemma \ref{pro5.1},  \ref{pro5.6} and \ref{lem5.4} we have that $\Span\{ a_{(m)}b|a\in E_1, b\in E_2\}=E_{\alpha}\subset L^k_{\alpha}$ an irreducible component of $E_1\otimes E_2$ because of the $\lieg$-action. 



Let $l:=k-(\alpha_2, \theta)+1$.  We have $\forall a\in E_1$ and $b$ the higuest weight of $E_2$
\begin{align*}
0&=[{a}_{\lambda}J_{e_{\theta}}^{l}b]=J_{e_{\theta}}[{a}_{\lambda}J_{e_{\theta}}^{l-1}b]+[a_{\lambda}J_{e_{\theta}}]J_{e_{\theta}}^{l-1}b-\int [{(e_{\theta}a)}_{\lambda}J_{e_{\theta}}^{l-1}b]d\mu
\end{align*}
The last term on the right-hand side has the highest power of $\lambda$. Hence, repeating  the last step $l$-times we found for $m=h-h_{a}-h_{b}-1$ 
\[0=[{a}_{\lambda}J_{e_{\theta}}^{l}b]=\lambda^{(m+l)}(-1)^l((e_{\theta})^{l}a)_{(m)}b +\lambda^{(m+l-1)}\cdots \, .\]
Therefore $((e_{\theta})^{l}a)_{(m)}b=0$. If \eqref{gw} is not satisfied then using the irreducible $\lieg$-action we have $a_{(m)}b=0$ for all $a\in E_1, b\in E_2$. 
\end{proof}

\subsection*{Acknowledgments} I am grateful to T. Arakawa, B. Bakalov,  R. Heluani, L. Topley, J. Van Ekeren by his patient hearing some ideas on this work.  Also, I want to thank T. Creutzig, D. Ridout for explain me some of their results.  This work was done at University of Bath and during at visit to IMPA  on April 2023, I am grateful to this institutions. The author was supported by UK Research and Innovation grant MR/S032657/1.

\bibliographystyle{amsalpha}

\begin{thebibliography}{DHVW}

\bibitem[Ad]{Ad}
Adamović, D. 
\textit{Realizations of Simple Affine Vertex Algebras and Their Modules: The Cases $\hat{\mathfrak{sl}}_2$ and $\hat{\mathfrak{osp}}(1,2)$}.
Commun. Math. Phys. 366, 1025–1067 (2019). 


\bibitem[A]{A}
Arakawa, T.:
\textit{A remark on the $C_2$-cofiniteness condition on vertex algebras}.
Math. Z. \textbf{270}, 559–575 (2012).

\bibitem[A2]{A2}
Arakawa, T.:
\textit{Associated varieties of modules over Kac–Moody algebras and $C_2$-Cofiniteness of W-algebras} 
Int. Math. Res. Not.,  22, 11605–11666 (2015).



 
\bibitem[BK]{BK}
Bakalov, B., Kac, V.G.:
\textit{Field algebras}.
Int. Math. Res. Not., \textbf{3}, 123-159 (2003)


\bibitem[BK2]{BK2}
Bakalov, B., Kac, V.G.:
\textit{Generalized vertex algebras}.
In: ``Lie theory and its applications in physics VI,'' 3--25, ed. V.K. Dobrev et al., 
Heron Press, Sofia, 2006; math.QA/0602072



\bibitem[BV]{BV}
Bakalov, B., Villarreal, J.:
\textit{Logarithmic vertex algebras}.
arxiv.org 2107.10206v2 (2021)

\bibitem[BV2]{BV2}
Bakalov, B., Villarreal, J.:
\textit{Logarithmic vertex algebras and non-local poisson vertex algebras}.
arxiv.org 2107.10206v2 (2021)

\bibitem[BPZ]{BPZ} 
Belavin,  A.A., Polyakov, A.M., Zamolodchikov, A.B.:
\textit{Infinite conformal symmetry in two-dimensional quantum field theory}.
Nuclear Phys. B \textbf{241}, 333--380 (1984)


\bibitem[Bo]{Bo}
Borcherds, R.E.:
\textit{Vertex algebras, Kac--Moody algebras, and the Monster}.
Proc. Nat. Acad. Sci. USA \textbf{83}, 3068--3071 (1986)


\bibitem[CR]{CR}
Creutzig, T., Ridout, D.:
\textit{ Modular Data and Verlinde Formulae for Fractional Level WZW Models I}.
 Nucl. Phys. B 865, 83 (2012) .
 
 \bibitem[CR2]{CR2}
Creutzig T., Ridout, D.:
\textit{Logarithmic conformal field theory: beyond an introduction}. 
J. Phys. A \textbf{46}, 494006, 72 pp. (2013)










\bibitem[DL]{DL}
Dong, C., Lepowsky, J.:
\textit{Generalized vertex algebras and relative vertex operators}.
Progress in Math., 112, Birkh\"auser Boston, 1993

\bibitem[DMS]{DMS}
Di Francesco, P., Mathieu, P., S\'en\'echal, D.:
\textit{Conformal field theory}.
Graduate Texts in Contemporary Physics, Springer--Verlag, New York, 1997


\bibitem[FB]{FB}
Frenkel, E., Ben-Zvi, D.:
\textit{Vertex algebras and algebraic curves}.
Math. Surveys and Monographs, 88,
Amer. Math. Soc., Providence, RI, 2001; 2nd ed., 2004



\bibitem[FF]{FF}
Feigin, B.L.  and Fuchs, D.B.:
\textit{ Verma modules over the Virasoro algebra},
 Lect. Notes Math. 1060, 230-245, (1984).
 
 \bibitem[FF1]{FF1}
Feigin, B.L.  and Fuchs, D.B.:
\textit{ Cohomology of some nilpotent subalgebras of Virasoro algebra and affine Kac-Moody Lie algebras},
J. Geom. Phys. 5, 209, (1988).

\bibitem[FHL]{FHL}
Frenkel, I.B., Huang,  Y., Lepowsky, J.,  Meurman, A.:
\textit{On Axiomatic Approaches to Vertex Operator Algebras and Modules}.
Mem Am Math Soc., 104, 1993

\bibitem[GW]{GW}
Gepner D., Witten E.:
\textit{String Theory on Group Manifolds}
Nucl.Phys.B 278, 493-549, (1986).











\bibitem[K]{K}
Kac, V.G.:
\textit{Vertex algebras for beginners}. 
University Lecture Series, 10, 
Amer. Math. Soc., Providence, RI, 1996; 2nd ed., 1998

\bibitem[K1]{K1}
Kac, V.G.:
\textit{Infinite dimensional Lie algebras}. 
Cambridge university press; 3nd ed., 1990


\bibitem[KRR]{KRR}
Kac, V.G., Raina, A.K., Rozhkovskaya, N.:
\textit{Bombay lectures on highest weight representations of infinite dimensional Lie algebras}. 
2nd ed., Advanced Ser. in Math. Phys., 29. 
World Sci. Pub. Co. Pte. Ltd., Hackensack, NJ, 2013















\bibitem[KZ]{KZ}
Knizhnik, V.G. , Zamolodchikov, A.B.:
\textit{Current Algebra and Wess-Zumino Model in Two-Dimensions}
Nucl.Phys.B 247, 83-103,  (1984).

\bibitem[Li3]{Li3}
Li, H.:
\textit{Abelianizing vertex algebras}. Comm. Math. Phys., Vol. \textbf{259}, No. 2, pp. 391–411, 2005.

\bibitem[Mi]{Mi}
Milas, A.:
\textit{Weak modules and logarithmic intertwining operators for vertex operator algebras}. 
(Charlottesville, VA, 2000), 201–225,
Contemp. Math. 297, Amer. Math. Soc., Providence, RI, 2002.



\end{thebibliography}

\end{document}